\documentclass[reqno,12pt]{amsart}

\usepackage{pdfpages}

\usepackage{epsf}
\usepackage{graphics}
\usepackage{graphicx}
\usepackage{amssymb}
\usepackage{amsmath}

\date{}

\theoremstyle{plain}
\newtheorem{theorem}{Theorem}
\newtheorem{corollary}{Corollary}
\newtheorem{proposition}{Proposition}
\newtheorem{lemma}{Lemma}
\newtheorem{rem}{Remark}

\theoremstyle{definition}

\theoremstyle{remark}

\def\N{{\mathbb N}}

\title{Combinatorial cusp count and clover invariants}

\author[Sebastian Baader]{Sebastian Baader}
\address{Mathematisches Institut, Universit\"at Bern, Sidlerstrasse 5, CH-3012 Bern, Switzerland}
\email{sebastian.baader@unibe.ch}

\author[Masaharu Ishikawa]{Masaharu Ishikawa}
\address{Faculty of Economics, Keio University, 4-1-1, Hiyoshi,
Kouhoku, Yokohama, Kanagawa 223-8521, Japan}
\email{ishikawa@keio.jp}

\begin{document}
\begin{abstract} We construct efficient topological cobordisms between torus links and large connected sums of trefoil knots. As an application, we show that the signature invariant $\sigma_\omega$ at $\omega=\zeta_6$ takes essentially minimal values on torus links among all concordance homomorphisms with the same normalisation on the trefoil knot.
\end{abstract}


\maketitle

\section{Introduction}

The topic of this note is motivated by the following question, already studied by Lefschetz~\cite{Lef}: how many simple cusps can a complex plane curve of degree~$d$ have? Here a simple cusp is locally described by the equation $y^2=x^3$. The answer is of order about~$\alpha d^2$, with a constant $\alpha$ known to lie in the interval $(\frac{29}{100},\frac{31}{100})$, as explained in the beautiful overview by Greuel and Shustin~\cite{GS}. Generically, a complex plane curve of degree~$d$ with~$N$ simple cusps gives rise to a smooth cobordism between the link at infinity - a torus link of type $T(d,d)$ - and the connected sum of~$N$ trefoil knots~$3_1$, the knot associated with the simple cusp. We study the following topological analogue of the above question: what is the locally flat topological cobordism of lowest complexity between a torus link of type $T(m,n)$ and the connected sum of~$N$ trefoil knots, denoted by~$3_1^N$? We consider the topological cobordism distance $d_\chi(L,L')$ between two links $L,L' \subset S^3$, defined as the minimal number of $1$-handles of a locally flat topological cobordism $C \subset S^3 \times [0,1]$ between~$L$ and~$L'$, consisting of connected components intersecting both $L$ and $L'$ (not to be confused with the smooth version of the cobordism distance introduced in~\cite{B}). In order to state our main result, we introduce the following variant of the Levine-Tristram signature function $\sigma_\omega(L)$ of a link~$L$ (see~\cite{Lev,T}) at $\omega=e^{\frac{2 \pi i}{6}}$:
$$\sigma_6(L)=\lim_{\epsilon \to 0+} \sigma_{e^{\frac{2 \pi i}{6}+\epsilon}}(L).$$

Unlike $\sigma_{e^{\frac{2 \pi i}{6}}}(L)$, $\sigma_6(L)$ provides a lower bound on the topological $4$-genus of $L$, even if the Alexander polynomial of~$L$ vanishes at $t=e^{\frac{2 \pi i}{6}}$. In particular, we have $\sigma_6(3_1)=2$, an important fact  for our purpose.

\begin{theorem}
\label{trefoil}
There exist constants $a,b,c>0$ with the following property.
For all $m,n,N \in \N$ with $N \geq \frac{7}{24}mn$:
$$|d_\chi(T(m,n),3_1^N)+\sigma_6(T(m,n))-\sigma_6(3_1^N)| \leq am+bn+c.$$
\end{theorem}

The value of $\sigma_6(T(m,n))$ is easy to extract from the work of Gambaudo and Ghys on the signature function on braid groups. Indeed, Proposition~5.2 in~\cite{GG} implies that the function $n \mapsto \sigma_6(T(m,n))$ is a quasimorphism of slope $\frac{5}{18}$, provided~$m$ is divisible by~$6$. This implies $\sigma_6(T(m,n)) \approx \frac{5}{18}mn$, up to an affine error in $m$ and $n$, for all $m,n \in \N$. This fact has an important consequence concerning a large class of concordance invariants. We define a clover invariant to be an additive link invariant~$\rho$ with the following two properties:
\begin{enumerate}
\item[(i)] $\rho(3_1)=2$,
\item[(ii)] $|\rho(L_1)-\rho(L_2)| \leq d_\chi(L_1,L_2)$, for all links~$L_1,L_2$.
\end{enumerate}
The second item implies $|\rho(K)| \leq 2g_4(K)$ for all knots~$K$, where $g_4(K)=\frac{1}{2}d_\chi(K,O)$ denotes the (locally flat) topological $4$-genus of~$K$, i.e. half the cobordism distance between~$K$ and the trivial knot~$O$.
As a consequence, $\rho$ vanishes on topologically slice knots.  Moreover, additivity implies that $\rho$ is a topological concordance invariant. An important family of clover invariants is given by the Levine-Tristram signature invariants $\sigma_{e^{2 \pi i \theta}}$ associated with $\theta \in (\frac{1}{6},\frac{1}{2}]$, and the limit invariant $\sigma_6$ defined above.

\begin{corollary}
\label{clover}
There exist constants $A,B,C>0$, so that the following inequality holds for all clover invariants $\rho$, and for all $m,n \in \N$:
$$\rho(T(m,n)) \geq \frac{5}{18}mn-Am-Bn-C.$$
\end{corollary}

The discussion after Theorem~\ref{trefoil} shows that the quadratic part of the lower bound, $\frac{5}{18}mn$, is sharp, since $\rho=\sigma_6$ is a clover invariant. In summary, the restriction of the invariant $\rho=\sigma_6$ to torus links is essentially dominated by every clover invariant.

It is easy to extract explicit values for the constants appearing in Theorem~\ref{trefoil} and Corollary~\ref{clover}. A careful inspection of the proofs shows that the constants $a,b$ and $A,B$ can be chosen to be about $20$, while $c$ and $C$ can be chosen to be about $200$. 

The proof of Theorem~\ref{trefoil} consists of two major steps, which we present in the following two sections. First, a rather involved construction of minimal cobordisms between $6$-strand torus links and large connected sums of trefoil knots. This is motivated by a result on the cobordism distance between closed positive $3$-braids and connected sums of trefoil knots~\cite{BR}. Second, a cabling construction which yields almost minimal cobordisms between general torus links and large connected sums of trefoil knots. The second step makes essential use of McCoy's twisting method~\cite{MC}. The proof of Corollary~\ref{clover} is short and simple; we present it in the last section.

\section*{Acknowledgements}

The first author is grateful to the Keio University for providing an excellent research environment during his stay in Tokyo. He also thanks Peter Feller for enlightening discussions on cusps in algebraic curves. The second author is supported by JSPS KAKENHI Grant numbers JP23K03098 and JP23H00081.

\section{Torus links with $6$ strands}

In this section we derive an almost precise expression for the topological cobordism distance between $6$-strand torus links and large connected sums of trefoil knots. Here and throughout this paper, we make use of the fact that the cobordism distance $d_\chi(L_1,L_2)$ is bounded below by the difference $|\sigma_6(L_1)-\sigma_6(L_2)|$. This is true, since $\sigma_6$ is a limit of Levine-Tristram signature invariants $\sigma_\omega$, and the lower bound holds for all $\sigma_\omega$ associated with non-algebraic numbers $\omega \in S^1$~\cite{P}.

\begin{proposition}
\label{sixstrand}
For all $m,n \in \N$ with $n \geq \frac{5}{3}m$:
$$d_\chi(T(6,m),3_1^n)=\sigma_6(3_1^n)-\sigma_6(T(6,m))+E(m,n),$$
where $E(m,n)$ is a globally bounded error term.
\end{proposition}

A direct application of Proposition~5.2 (for $\theta=\frac{1}{6}$) and Remark~1 in~\cite{GG} shows
$\sigma_6(T(6,m))=\frac{5}{3}m+E(m)$, where $E(m) \leq 12$. Therefore, in order to prove Proposition~\ref{sixstrand}, we need to construct a connected cobordism with Euler characteristic of absolute value about $2n-\frac{5}{3}m$ between the two links $T(6,m)$ and $3_1^n$. This cobordism will in fact be a sequence of smooth saddle moves and smooth concordances, so that Proposition~\ref{sixstrand} remains true in the smooth category.

As a preparation, we derive an algebraic statement about the third power of the central element $(abc)^{4}$ in the braid group~$B_4$. Here, for simplicity, we denote the standard generators of $B_4$ by $a,b,c$ instead of the commonly used $\sigma_1,\sigma_2,\sigma_3$. Let $\alpha,\beta \in B_4$ be braids represented by words in the generators $a,b,c$. We say that $\beta$ is related to $\alpha$ by a negative $t_3$-move, if $\alpha$ is obtained from $\beta$ by removing the third power of any of the standard generators, anywhere in the braid word $\beta$. As observed in~\cite{BR} (Lemma~1), the link $\hat{\beta}$ and the connected sum of links $\hat{\alpha} \# 3_1$ are then related by a single saddle move, in particular
$$d_\chi(\hat{\beta},\hat{\alpha} \# 3_1)=1.$$

\begin{lemma}
\label{fourstrand}
The braid $\beta=a^{-3}c^{-3}(abc)^{12} \in B_4$ can be transformed into the trivial braid by a sequence of $10$ negative $t_3$-moves.
\end{lemma}

The proof just below also implies the following, more natural, statement, which was already known to Coxeter~\cite{Cox}: the braid $(abc)^{12} \in B_4$ can be transformed into the trivial braid by a sequence of $12$ negative $t_3$-moves. However, we will need the more specific formulation of Lemma~\ref{fourstrand} in the proof of Proposition~\ref{sixstrand}.

\begin{proof}[Proof of Lemma~\ref{fourstrand}]
We use the following algebraic identity, which is a variation of the well-known equality $(abc)^{12}=(a^2cb)^9$ in $B_4$ stated in~\cite{Cox}:
$$(abc)^{12}=(a^2cba^3cb)^4=\gamma.$$
Figure~1 shows an isotopy between the braid $(a^2cba^3cb)^4$ and a $4$-braid which is easy to identify as the third power of a full twist on four strands, i.e. 
$(abc)^{12}$.
\begin{figure}[htbp]
\includegraphics[scale=1.1, bb=155 589 437 713]{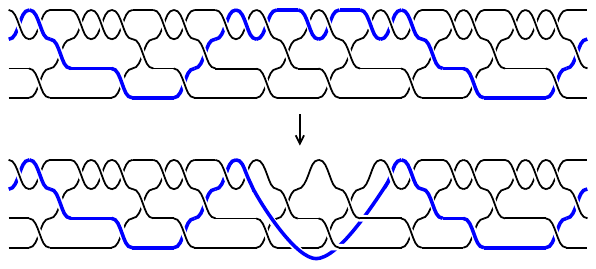}
\caption{$(a^2cba^3cb)^4=(abc)^{12}$}
\label{figure1}
\end{figure}
After applying $4$ negative $t_3$-moves to~$\gamma$, we obtain the braid
$$(a^2(cb)^2)^4=c^2(a^2bc^3)^3a^2bc.$$
Another $3$ negative $t_3$-moves transform the latter into 
$$c^2(a^2b)^3a^2bc=c^2(a^3b)^3c=\delta.$$
Here we use the identity $(a^2b)^4=(a^3b)^3$.
Another $3$ negative $t_3$-moves (removing the second and third instance of $a^3$, then $b^3$) transform $\delta$ into $c^2a^3c=c^3a^3$.
We have just seen that the positive braid $(abc)^{12}$ can be transformed into the positive braid $c^3a^3$ by a sequence of $4+3+3=10$ negative $t_3$-moves. Therefore, the braid $\beta=a^{-3}c^{-3}(abc)^{12} \in B_4$ can be transformed into the trivial braid by a sequence of $10$ negative $t_3$-moves.
\end{proof}

\begin{proof}[Proof of Proposition~\ref{sixstrand}]
We may assume $m=6k$, since every positive $6$-strand torus link is related to $T(6,6k)$ by a sequence of at most $15$ saddle moves, thus by a smooth cobordism of Euler characteristic at most $15$. This operation does not change the value $\sigma_6(T(6,m))$ by more than $15$. Furthermore, we need only consider the case $n=10k$, for the following reason:
for all $n'>n$,
$$d_\chi(3_1^{n'},3_1^{n})=2(n'-n)=\sigma_6(3_1^{n'})-\sigma_6(3_1^{n}).$$
Indeed, the two knots $3_1^{n}$, $3_1^{n'}$ are related by $n'-n$ crossing changes, thus by a smooth cobordism of Euler characteristic $2(n'-n)$. 
In the first step, we construct a smooth cobordism of small Euler characteristic between the link $T(6,6k)$ and the closure of the braid
$$(dced(bacb)^5a^3c^3)^{k-3},$$
where $a,b,c,d,e$ denote the standard generators of the braid group $B_6$.
For this, we view $T(6,6k)$ as a $2$-cable of $T(3,3k)$. In~\cite{BBL}, a special positive braid representing the link $T(3,3k)$ is derived, which depends on the parity of~$k$. We only present the odd case $k=2l+1$ here; the even one is virtually the same. The link $T(3,6l+3)$ is isotopic to the closure of the $3$-braid
$$(ba^4ba^3(ba^5)^{l-1})^2.$$
By replacing $a,b \in B_3$ by $bacb,dced\in B_6$, respectively, and introducing the correct framing of the $2$-cable in front, we obtain the following 6-braid representing the link $T(6,6k)=T(6,12l+6)$:
$$(ace)^{4l+2}(dced(bacb)^4dced(bacb)^3(dced(bacb)^5)^{l-1})^2.$$
The easiest way to check that the framing $(ace)^{4l+2}$ is indeed correct is by computing the total number of crossings, which should coincide with the crossing number $c(T(6,12l+6))=60l+30$. The precise location of the framing is not relevant; in particular, we may slide it along the core link $T(3,6l+3)$ and distribute it right after the brackets $(bacb)^5$. As a result, after smoothing a bounded number of crossings by saddle moves ($90$, to be precise), the above braid can be transformed into the braid
$$\beta=(dced(bacb)^5a^3c^3)^{2l-2}.$$

Now comes the second step: 
The braid~$\beta$ is easily identified as
$$(dced(bacb)^{-1}(bacb)^6a^3c^3)^{2l-2}=(dced(bacb)^{-1}a^{-3}c^{-3}(abc)^{12})^{2l-2},$$ 
since the $4$-braid $(bacb)^6$ is a $2$-cable of the $2$-braid $a^6$.

Thanks to Lemma~\ref{fourstrand}, the braid~$\beta$ can be reduced to the braid
$$\alpha=(dced(bacb)^{-1})^{2l-2}$$
by a sequence of $10 \cdot (2l-2)$ negative $t_3$-moves. As stated just before Lemma~\ref{fourstrand}, the two links $\hat{\beta}$ and $\hat{\alpha} \# 3_1^{20l-20}$ are thus related by a sequence of $20l-20$ saddle moves. Moreover, the link $\hat{\alpha}$ can be transformed into the a smoothly slice knot by a constant number of saddle moves, about ten in number. Indeed, after five suitable saddle moves, the link $\hat{\alpha}$ turns into the connected sum of links $L \# L$, where~$L$ is the closure of the braid $(dced(bacb)^{-1})^{l-1}$, see Figure~2. The latter is isotopic to its mirror image, so $L \# L$ is smoothly concordant to the trivial link with six components. Another five saddle moves transform the latter into the trivial knot.
\begin{figure}[htbp]
\includegraphics[scale=1.1, bb=208 556 386 712]{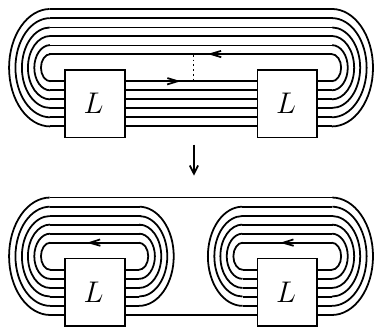}
\caption{Five saddle moves}
\label{figure2}
\end{figure}
As a consequence, the original link $T(6,12l+6)$ can be transformed into the connected sum of trefoil knots $3_1^{20l}$ by a sequence of about $20l$ saddle moves and link concordances, up to a bounded error. Keeping in mind $m=6k=12l+6$ and $n=10k=20l+10$, we get indeed
\begin{equation*}
\begin{split}
d_\chi(T(6,m),3_1^n) & =20l+C(m,n) \\
 & =2n-\frac{5}{3}m+10+C(m,n) \\
 & =\sigma_6(3_1^n)-\sigma_6(T(6,m))+E(m,n), \\
\end{split}
\end{equation*}
with globally bounded error terms $C(m,n),E(m,n)$.
\end{proof}

The above proof produces an explicit upper bound smaller than $200$ on the error term $E(n,m)$; this is far from optimal since we tried to keep the argument short.

\section{Twisting torus links}

The proof of Theorem~\ref{trefoil} relies on McCoy's twisting method~\cite{MC}. A null-homologous twist is an operation on oriented links that takes place around a disc that intersects an even number of strands of a link transversely, with equally many strands going in either direction. A positive (resp. negative) twist inserts a positive (resp. negative) full twist into these strands. As an example, the torus link $T(2k,2k)$ is related to the disjoint union of two torus links of type $T(k,2k)$ by a single negative twist. A special case of Theorem~1 in~\cite{MC} states that if an oriented knot~$K$ can be transformed into the trivial knot by a sequence of~$t$ positive and~$t$ negative null-homologous twists, then $g_4(K) \leq t$. It is the combination of the positive and negative twists that allows us to prove the following lemma, which is the second key ingredient in the proof of Theorem~\ref{trefoil}.

\begin{lemma}
\label{twisting}
For all $k,l \in \N$ coprime and $t \geq \frac{1}{2}(k-1)(l-1)$:
$$d_\chi(T(6k,6l),T(6,6kl) \# 3_1^t) \leq 2t+10.$$
\end{lemma}

There is an ambiguity in the meaning of the direct sum $T(6,6kl) \# 3_1^t$ in the above statement; we use the convention where all the trefoil summands are attached to the same component of the link $T(6,6kl)$.

\begin{proof}[Proof of Lemma~\ref{twisting}]
We start by observing that the link $T(6k,6l)$ is a $6$-cable of the torus knot $T(k,l)$ with framing $kl$. Indeed, all components of $T(6k,6l)$ have pairwise linking number $kl$. The knot $T(k,l)$ can be transformed into the trivial knot by a sequence of $t=\frac{1}{2}(k-1)(l-1)$ negative crossing changes. As a consequence, the link $T(6k,6l)$ can be transformed into the $kl$-framed $(6,0)$-cable of the trivial knot, i.e. into the torus link $T(6,6kl)$, by a sequence of~$t$ negative null-homologous twists (compare Section~5 in~\cite{MC}). In order to apply McCoy's $4$-genus bound, we need to consider knots rather than links. Let $K$ be the $0$-framed $(6,1)$-cable of the knot $T(k,l)$. By definition, the knot~$K$ is represented by the braid
$$(abcde)^{-1-6kl} \delta \in B_{6k},$$
where $\delta \in B_{6k}$ is the standard braid representing the torus link $T(6k,6l)$, and $a,b,c,d,e$ denote the first five standard generators of the braid group $B_{6k}$. Moreover, the knot~$K$ can be transformed into the trivial knot by a sequence of~$t$ negative null-homologous twists. In turn, the knot $K \# 3_1^{-t}$ can be transformed into the trivial knot by a sequence of~$t$ negative and~$t$ positive null-homologous twists, since we can remove one negative trefoil summand with each positive twist. As a consequence
$g_4(K \# 3_1^{-t}) \leq t$, hence
$$d_\chi(K,3_1^t) \leq 2t.$$
We are nearly done, since the link $T(6k,6l)$ and the link~$T(6,6kl) \# K$ are related by a sequence of just $10$ saddle moves:
\begin{equation*}
\begin{split}
d_\chi(T(6k,6l), T(6,6kl) \#  3_1^t) & \leq d_\chi(T(6,6kl) \# K,T(6,6kl) \#  3_1^t)+10 \\
 & = d_\chi(K,3_1^t)+10 \\
 &\leq 2t+10.
\qedhere
\end{split}
\end{equation*}
\end{proof}

Before we prove Theorem~\ref{trefoil}, we invoke again the formula of Gambaudo and Ghys for $\sigma_6(T(m,n))$ (Proposition 5.2 in~\cite{GG}). Their formula holds in fact for a homogenised version of the Levine-Tristram invariant denoted by $\text{Sign}_{e^{\frac{2 \pi i}{6}}}$. By Remark~1 in~\cite{GG}, the restriction of the latter to the braid group $B_m$ differs from the invariant $\sigma_{e^{\frac{2 \pi i}{6}}}$, and thus from our limit invariant $\sigma_6$, by a bounded error of size at most $2m$ (two times the braid index). We obtain the following estimate from their formula, valid for all $m$ divisible by six:
$$|\sigma_6(T(m,n))-\frac{5}{18}mn| \leq 2m.$$
Since we allow for an affine error in $m$ and $n$, we may use the approximate formula $\sigma_6(T(m,n)) \approx \frac{5}{18}mn$ for all $m,n \in \N$.

\begin{proof}[Proof of Theorem~\ref{trefoil}]
Let $m,n \in \N$. We may replace the link $T(m,n)$ by a link of the form $T(6k,6l)$ with $|m-6k| \leq 3$, $|n-6l| \leq 3$. This changes the value of $\sigma_6(T(m,n))$ and $d_\chi(T(m,n),3_1^N)$ by $3(m+n)$, at most. Therefore, in order to prove Theorem~\ref{trefoil}, we need to construct a connected cobordism with Euler characteristic of absolute value about $2N-\frac{5}{18}mn=2N-10kl$ between the two links $T(6k,6l)$ and $3_1^N$, for all $N \geq \frac{7}{24}mn=\frac{21}{2}kl$. For simplicity, we assume that $k,l$ are coprime. The general case is just a variation on this:
if $k,l$ are not coprime, we can transform the link $T(k,l)$ into a positive braid knot by smoothing at most $k$ crossings. As a consequence, the link $T(6k,6l)$ can be transformed into a $6$-cable of a positive braid knot by a sequence of at most $36k$ saddle moves.

We are finally in the position to put together the two main steps of the argument.
First, by Lemma~\ref{twisting},
$$d_\chi(T(6k,6l),T(6,6kl) \# 3_1^t) \leq 2t+10,$$
for all $t \geq \frac{1}{2}(k-1)(l-1)$.
Second, by Proposition~\ref{sixstrand},
$$d_\chi(T(6,6kl),3_1^n) \approx \sigma_6(3_1^n)-\sigma_6(T(6,6kl)) \approx 2n-10kl,$$
up to a globally bounded error term, for all $n \geq 10kl$.
Putting these two bounds together, and setting $N=t+n$ with $t \geq \frac{1}{2}kl$ and  $n \geq 10kl$, we obtain
\begin{equation*}
\begin{split}
d_\chi(T(6k,6l),3_1^N) &\leq d_\chi(T(6k,6l),T(6,6kl) \# 3_1^t)+d_\chi(T(6,6kl) \# 3_1^t,3_1^N) \\
 &=d_\chi(T(6k,6l),T(6,6kl) \# 3_1^t)+d_\chi(T(6,6kl),3_1^n)\\
 &\leq 2t+10+2n-10kl \approx 2N-10kl, \\
\end{split}
\end{equation*}
up to a globally bounded error term, for all $N \geq \frac{21}{2}kl$, as required.
\end{proof}

\section{A lower bound on clover invariants}

We consider a clover invariant, i.e. an additive link invariant~$\rho$ satisfying $\rho(3_1)=2$ and $|\rho(L_1)-\rho(L_2)| \leq d_\chi(L_1,L_2)$, for all links~$L_1,L_2$. The second property together with Theorem~\ref{trefoil} implies for all $N \geq \frac{7}{24}mn$:

\begin{equation*}
\begin{split}
|\rho(T(m,n))-\rho(3_1^N)| &\leq d_\chi(T(m,n),3_1^N) \\
 &\leq 2N-\sigma_6(T(m,n))+am+bn+c \\
 &\leq 2N-\frac{5}{18}mn+Am+Bn+C, \\
\end{split}
\end{equation*}
for suitable constants $A,B,C>0$. The last inequality holds thanks to the formula by Gambaudo and Ghys discussed in the paragraph after Theorem~\ref{trefoil}. This concludes the proof of Corollary~\ref{clover}, since the normalisation $\rho(3_1^N)=2N$ implies
$$\rho(T(m,n)) \geq \frac{5}{18}mn-Am-Bn-C.$$

\end{document}